\documentclass[reqno, 12pt]{amsart} 

\usepackage{fmtcount}
\usepackage[expansion=false]{microtype} 
\usepackage{amsfonts,amsthm,amsmath,amssymb,amscd,mathrsfs} 
\allowdisplaybreaks  
\usepackage{latexsym} 
\usepackage{xcolor, color} 
\usepackage[colorlinks=true,linkcolor=blue,citecolor=teal,urlcolor=black]{hyperref} 
\usepackage{graphicx}  
\usepackage{indentfirst} 

\usepackage{mathtools}
\DeclarePairedDelimiter{\norm}{\lVert}{\rVert}

\renewcommand{\MR}[1]{}

\usepackage{lmodern} 
\usepackage{soul}     
\usepackage{wasysym}  

\theoremstyle{plain} 
\newtheorem{Thm}{Theorem}[section] 
\newtheorem{Lem}[Thm]{Lemma}     
\newtheorem{Prop}[Thm]{Proposition}
\newtheorem{Cor}[Thm]{Corollary}
\theoremstyle{definition}

\theoremstyle{remark}
\newtheorem{Rem}[Thm]{Remark}
\numberwithin{equation}{section} 

\newcommand{\beq}{\begin{equation}}            
	\newcommand{\eeq}{\end{equation}}
\newcommand{\ben}{\begin{eqnarray}}         
	\newcommand{\een}{\end{eqnarray}}
\newcommand{\beno}{\begin{eqnarray*}}
	\newcommand{\eeno}{\end{eqnarray*}}

\newcommand{\alabel}{\stepcounter{equation}\tag{\theequation}\label}  
\newcommand{\R}{\mathbb{R}}
\newcommand{\curl}{\operatorname{curl}}
\newcommand{\diver}{\operatorname{div}}

\newcommand{\A}{\mathcal{A}}
\newcommand{\Q}{\mathcal{Q}}

\newcommand{\pv}{\operatorname{p.v.}}
\newcommand{\ee}{\mathrm{e}}
\newcommand{\Lr}[1]{\log(\ee+#1)}

\usepackage[
letterpaper,                 
textheight=8.35in,           
headsep=20pt,                
footskip=36pt,               
marginparwidth=0pt,          
marginparsep=0pt,            
left=0.75in, 
right=0.75in
]{geometry}

\title[New decay estimates and Liouville type theorems]{New decay estimates and Liouville type theorems for the 3D axisymmetric stationary Navier--Stokes equations}

\author{
	Wendong~Wang$^{1}$ \and
	Guoxu~Yang$^{1,*}$
}

\thanks{$^{1}$School of Mathematical Sciences, Dalian University of Technology, 
	Dalian 116024, China.}

\thanks{E-mail addresses: 
	wendong@dlut.edu.cn (Wendong Wang), 
	guoxu\_dlut@outlook.com (Guoxu Yang).}

\thanks{$^{*}$Author to whom correspondence should be addressed.}



\begin{document}

\begin{abstract}
	The Liouville problem for the three-dimensional stationary Navier--Stokes equations remains open, even for axisymmetric \(D\)-solutions. In this paper, we obtain two results based on decay in the cylindrical radial variable \(r=|x'|\). \\
	(i). Using a new pointwise Calderón--Zygmund estimate adapted to cylindrical geometry, we improve the decay estimates of Carrillo--Pan--Zhang (2020, JFA) and prove
	\[
	|\nabla u_r|+|\nabla u_z|
	\lesssim r^{-5/4}[\log(\mathrm e+r)]^{5/4},
	\quad
	|\omega_r|+|\omega_z|
	\lesssim r^{-9/8}[\log(\mathrm e+r)]^{9/8}, \quad r\gg1.
	\]
	(ii). We develop a new approach to Liouville theorems that improves the axisymmetric criteria of Wang (2019, JDE) and Zhao (2019, Nonlinear Anal.). Without any symmetry assumption, we show that a \(D\)-solution is trivial if one of the following holds:
	\[ (\mathrm a).\,\sup_{{|x'|=r,\, z\in\mathbb R}}
	|u(x',z)|
	\leq
	Cr^{-2/3}[\log(\mathrm e+r)]^{-\gamma}; \quad (\mathrm b).\,
	\sup_{{|x'|=r,\, z\in\mathbb R}}
	|\omega(x',z)|
	\leq
	Cr^{-5/3}[\log(\mathrm e+r)]^{-\gamma},
	\]
	for $r\geq1$, where $\gamma>1/3$.
\end{abstract}

\maketitle


\section{Introduction}

\subsection{Background and previous results}

We consider the steady incompressible Navier--Stokes equations in $\R^3$:
\begin{equation}
	\left\{
	\begin{aligned}
		-\Delta u+(u\cdot\nabla)u+\nabla p&=0,\\
		\diver u&=0,\\
		\lim_{|x|\to\infty}u(x)&=0.
	\end{aligned}
	\right.
	\label{eq:NS}
\end{equation}
A smooth solution of \eqref{eq:NS} satisfying
\begin{equation}
	\int_{\R^3}|\nabla u|^2\,dx<\infty
	\label{eq:Dsolution}
\end{equation}
is called a $D$-solution. The finite Dirichlet integral \eqref{eq:Dsolution} is natural in the study of stationary viscous flows and goes back to Leray \cite{L1933}. By the Sobolev inequality, every $D$-solution satisfies
\begin{equation}
	u\in L^6(\R^3).
	\label{eq:L6}
\end{equation}

Liouville theorems relate the behavior of a stationary flow at infinity to its global rigidity. A fundamental open problem asks whether every three-dimensional \(D\)-solution of \eqref{eq:NS} must be identically zero. A classical sufficient condition is \(u\in L^{9/2}(\R^3)\); see \cite{G2011}. Chae established a distinct scale-invariant criterion \(\Delta u\in L^{6/5}(\R^3)\), whereas Chae--Wolf obtained a logarithmic refinement of Galdi's \(L^{9/2}\)-criterion.
Seregin \cite{S2016} introduced a criterion in \(BMO^{-1}(\R^3)\), Kozono--Terasawa--Wakasugi \cite{KTW2017} considered vorticity-decay and Lorentz-space conditions, and Seregin--Wang \cite{SW2020} developed corresponding annular criteria; see also \cite{T2021} and the references therein. Further criteria based on the head pressure \(Q=\frac12|u|^2+p\) were obtained by Chae \cite{C2021}, who proved triviality under relative-decay assumptions comparing \(u\) or \(p\) with \(Q\) at infinity. Bang--Yang \cite{BY2025} employed Saint--Venant estimates to derive Liouville theorems from the growth of the \(L^s\)-mean oscillation of a potential function of the velocity and from the relative decay of \(Q\) and \(|u|^2\). Chae \cite{C2025} subsequently established further head-pressure criteria through Osgood-type conditions, level-set arguments, and the coarea formula. Coiculescu--Yang \cite{CY2026} introduced a capsule-based approach and obtained improved conditional Liouville theorems under sufficiently slow sublinear growth assumptions on antiderivatives of the velocity, including vector potentials and line integrals. More recently, Chae \cite{C2026} considered D-solutions whose head pressure \(Q\), normalized to vanish at infinity, satisfies the lower power-law bound \(|Q(x)|\geq C\|Q\|_{L^\infty}|x|^{-\alpha}\) for all sufficiently large \(|x|\). He proved that \(u\equiv0\) if \(|u(x)|=O(|x|^{-\beta})\) for some \(\beta\geq\frac{\alpha}{2}\), whereas \(u\) must be constant if \(|\nabla Q(x)|=O(|x|^{-\beta})\) for some \(\beta\geq2\alpha\).

Symmetry-restricted flows form a natural intermediate setting in which genuinely three-dimensional structures remain accessible to sharper analysis. For helically symmetric flows, Han--Wang--Xie \cite{HWX2026} proved that every bounded smooth solution of the steady Navier--Stokes equations in \(\R^3\) is a constant vector, using the helical identity and axial periodicity to establish a Saint--Venant-type estimate. Axisymmetric flows retain genuinely three-dimensional phenomena, including swirl, while their cylindrical structure permits refined pointwise estimates. In the no-swirl case \(u_\theta\equiv0\), Liouville theorems were established by Koch--Nadirashvili--Seregin--{\v S}ver\'ak \cite{KNSS2009} and Korobkov--Pileckas--Russo \cite{KPR2015}. For axisymmetric flows with swirl, Chae--Weng \cite{CW201601} obtained several sufficient conditions for triviality. A complementary line of research concerns the asymptotic behavior of axisymmetric \(D\)-solutions. Pointwise decay estimates were developed by Choe--Jin \cite{CJ2009}, Weng \cite{W2018}, and Carrillo--Pan--Zhang \cite{CPZ2020}; in particular, Carrillo--Pan--Zhang derived logarithmic decay bounds for the velocity and the vorticity components. Within the framework of pointwise cylindrical decay, Seregin \cite{S2018} proved that a bounded smooth axisymmetric solution must vanish if, uniformly in the axial variable, the velocity satisfies \(O(r^{-\alpha})\) for some $\alpha>\frac{15-\sqrt{33}}{12}$. Subsequently, Wang \cite{W2019} and Zhao \cite{Z2019} independently improved this velocity-decay criterion by showing that an axisymmetric \(D\)-solution must vanish if the velocity satisfies \(O(r^{-\alpha})\) for some \(\alpha>\frac23\), or if the full vorticity satisfies \(O(r^{-\beta})\) for some \(\beta>\frac53\), uniformly in the axial variable.

Motivated by these studies, we study Liouville-type theorems for the axisymmetric Navier--Stokes equations in this paper.

\subsection{Main results and ideas}

Fix a direction in $\R^3$ and write
\begin{equation}
	x=(x',z)\in\R^2\times\R,
	\qquad
	r=|x'|.
	\label{eq:splitting}
\end{equation}
Let $L(s):=\log(\ee+s)$ with $s\geq0$. We use the standard cylindrical components $u_r,u_\theta,u_z$ of an axisymmetric velocity field and $\omega_r,\omega_\theta,\omega_z$ of its vorticity $\omega=\curl u$. When $\nabla$ acts on one of these scalar components, it means the meridional gradient $\nabla_{r,z}$; the precise convention is given in Subsection \ref{subsec:notation}. For $R>0$, define the cylindrical annulus
\begin{equation}
	\A_R:=\{(x',z)\in\R^2\times\R:R\leq |x'|\leq2R\}.
	\label{eq:cylindrical-annulus}
\end{equation}

Our first main result is the following.

\begin{Thm}
	\label{thm:main}
	Let $u$ be a smooth axisymmetric solution of \eqref{eq:NS}. Assume that there exists $M>0$ such that
	\begin{equation}
		\sup_{R\geq1}\int_{\A_R}\bigl(|\nabla u(x)|^2+|u(x)|^6\bigr)\,dx\leq M.
		\label{eq:annular-assumption}
	\end{equation}
	Then there exist \(R_0=R_0(M,u)\geq 2\) and \(C=C(M)>0\) such that, for every $r\geq R_0$ and every $z\in\R$,
	\begin{equation}
		|u(r,z)|\leq C r^{-1/2}[\Lr r]^{1/2},
		\label{eq:velocity-final}
	\end{equation}
	\begin{equation}
		|\omega_\theta(r,z)|+|\nabla\omega_\theta(r,z)|\leq C r^{-5/4}[\Lr r]^{3/4},
		\label{eq:wtheta-final}
	\end{equation}
	\begin{equation}
		|\nabla u_r(r,z)|+|\nabla u_z(r,z)|\leq C r^{-5/4}[\Lr r]^{5/4},
		\label{eq:grad-b-final}
	\end{equation}
	and
	\begin{equation}
		|\omega_r(r,z)|+|\omega_z(r,z)|\leq C r^{-9/8}[\Lr r]^{9/8}.
		\label{eq:wm-final}
	\end{equation}
\end{Thm}

\begin{Rem}	
	The velocity estimate \eqref{eq:velocity-final} is due to Choe--Jin \cite{CJ2009} and Weng \cite{W2018}, while the bound for $\omega_\theta$ in \eqref{eq:wtheta-final} is due to Carrillo--Pan--Zhang \cite{CPZ2020}. Choe--Jin \cite{CJ2009} first proved
	\[
	|u_r|+|u_z|\lesssim r^{-\frac12} (\log r)^\frac12,\quad
	|u_\theta|\lesssim {r^{-\frac38}}{(\log r)^{\frac18}},\quad
	|\omega_\theta|\lesssim r^{-\frac78},
	\]
	which was improved by Weng \cite{W2018}:
	\[
	|u|\lesssim r^{-\frac12} (\log r)^\frac12,\quad
	|\omega_\theta|\lesssim r^{-\left(\frac{19}{16}\right)^-},\quad
	|\nabla u_r|+|\nabla u_z|\lesssim r^{-\left(\frac98\right)^-},\quad
	|\omega_r|+|\omega_z|\lesssim r^{-\left(\frac{67}{64}\right)^-}.
	\]
	Later, Carrillo--Pan--Zhang \cite{CPZ2020} further improved the decay rates of the vorticity and the gradients of the meridional velocity components to
	\[
	|\omega_\theta|\lesssim {r^{-\frac54}}{(\log r)^{\frac34}},\qquad
	|\nabla u_r|+|\nabla u_z|\lesssim {r^{-\frac54}}{(\log r)^{\frac74}},\qquad
	|\omega_r|+|\omega_z|\lesssim {r^{-\frac98}}{(\log r)^{\frac{11}{8}}}.
	\]
	In this paper, the new estimates are \eqref{eq:grad-b-final} and \eqref{eq:wm-final}. Thus the logarithmic exponents are improved according to
	\[
	\frac74\longrightarrow\frac54,\qquad \frac{11}{8}\longrightarrow\frac98.
	\]
\end{Rem}

The main new ingredient is a new Calderón--Zygmund estimate for an axisymmetric azimuthal field. Rotational packing improves a local $L^2$ norm, while the cancellation of the Calder\'on--Zygmund kernel and the Poincar\'e inequality are used on each intermediate shell. In the notation of Lemma \ref{lem:refined-CZ}, the resulting estimate is
	\[
	|TF(x_0)|\lesssim A R^{-a}[\Lr R]^\beta+B_*R_*^2R^{-2}+R^{-1/2}\Lr R\,\|\nabla_{r,z}f\|_{L^2(\Q_R'(z_0))}.
	\]
A localized vorticity energy estimate gives
	\[
	\|\nabla\omega_\theta\|_{L^2(r\sim R,\,|z-z_0|\lesssim R)}\lesssim R^{-3/4}[\Lr R]^{1/4}.
	\]
These two estimates yield the logarithmic power $5/4$ in \eqref{eq:grad-b-final}. The estimate \eqref{eq:wm-final} then follows from a scaled Brezis--Gallouet inequality \cite{BG1980} for $(\omega_r,\omega_z)$.

\begin{Cor}
	The conclusions of Theorem \ref{thm:main} hold for every smooth axisymmetric $D$-solution of \eqref{eq:NS}.
\end{Cor}


We next consider general $D$-solutions without any symmetry assumption. Using the fixed splitting in \eqref{eq:splitting}, define the radial envelope
\begin{equation}
	H(r):=\sup_{|x'|=r,\,z\in\R}|u(x',z)|.
	\label{eq:H}
\end{equation}
The first step is the following Liouville criterion.

\begin{Prop}
	\label{thm:envelope}
	Let $u$ be a smooth solution of \eqref{eq:NS} satisfying \eqref{eq:Dsolution}. If
	\begin{equation}
		\int_1^\infty rH(r)^3\,dr<\infty,
		\label{eq:envelope-condition}
	\end{equation}
	then $u\equiv0$.
\end{Prop}

Proposition \ref{thm:envelope} immediately gives the following.

\begin{Thm}
	\label{thm:velocity}
	Let $u$ be a smooth solution of \eqref{eq:NS} satisfying \eqref{eq:Dsolution}. Suppose that, for some $\gamma>\frac13$,
	\begin{equation}
		\sup_{{|x'|=r ,\, z\in\R}}|u(x',z)|\leq \frac{C}{r^{2/3}[\log(\ee+r)]^\gamma},
		\qquad r\geq1.
		\label{eq:velocity-endpoint}
	\end{equation}
	Then $u\equiv0$.
\end{Thm}

The corresponding vorticity criterion is our second main result.

\begin{Thm}
	\label{thm:vorticity}
	Let $u$ be a smooth solution of \eqref{eq:NS} satisfying \eqref{eq:Dsolution}, and let $\omega=\curl u$. Suppose that, for some $\gamma>\frac13$,
	\begin{equation}
		\sup_{{|x'|=r,\, z\in\R}}|\omega(x',z)|\leq \frac{C}{r^{5/3}[\log(\ee+r)]^\gamma},
		\qquad r\geq1.
		\label{eq:vorticity-endpoint}
	\end{equation}
	Then $u\equiv0$.
\end{Thm}


\begin{Rem}
	Theorem \ref{thm:vorticity} should be compared with the cylindrical criteria of Seregin \cite[Corollary 1.9]{S2018}, Wang \cite[Theorem 1.3]{W2019} and Zhao \cite[Theorem 1.2]{Z2019}. Relative to these results, the strict algebraic condition \(\beta>\frac53\) is replaced by the critical exponent \(\beta=\frac53\) together with a logarithmic gain, and no axisymmetry assumption is imposed. The hypothesis is a pointwise bound on the full vorticity, uniform on the cylinders $\{(x',z)\in\mathbb R^2\times\mathbb R:|x'|=r\}$. In particular, it requires no additional decay as \(|z|\to\infty\) with \(r\) fixed. This differs both from isotropic vorticity conditions formulated in terms of the full distance \(|x|\), such as those in \cite{KTW2017}, and from the anisotropic velocity-integrability criteria in \cite{C2023,ZZ2025}. It is also distinct from the head-pressure criteria in \cite{C2026}. 
\end{Rem}

\begin{Rem}
	The logarithmic factor in \eqref{eq:vorticity-endpoint} can be relaxed to a finite hierarchy of iterated logarithms. 
\end{Rem}

The proof of Theorem \ref{thm:vorticity} has two steps. First, after integrating the three-dimensional Biot--Savart kernel in the axial variable, a two-dimensional logarithmically weighted Riesz-potential estimate transfers the decay of $\omega$ to
	\[
	H(r)\lesssim r^{-2/3}[\log(\ee+r)]^{-\gamma}.
	\]
Second, a Bogovskii correction gives a pressure-free Caccioppoli inequality. The condition $\gamma>\frac13$ is exactly what makes the cubic radial flux integrable:
	\[
	\int_1^\infty rH(r)^3\,dr<\infty.
	\]

\subsection{Notation and organization}
\label{subsec:notation}

For $R>0$ and $x_0\in\R^3$, denote
\begin{equation}
	B_R(x_0):=\{x\in\R^3:|x-x_0|<R\},
	\qquad
	B_R:=B_R(0),
	\qquad
	A_R:=B_{2R}\setminus\overline{B_R}.
	\label{eq:geometric-sets}
\end{equation}
Thus $A_R$ is a spherical annulus, while $\A_R$ in \eqref{eq:cylindrical-annulus} is a cylindrical annulus about the fixed axis in \eqref{eq:splitting}. For the axisymmetric solutions, we use cylindrical coordinates
\[
x_1=r\cos\theta,\qquad x_2=r\sin\theta,\qquad z=x_3,
\]
with $e_r=(\cos\theta,\sin\theta,0)$, $e_\theta=(-\sin\theta,\cos\theta,0)$ and $e_z=(0,0,1)$. An axisymmetric velocity field and its vorticity are written as
\[
u=u_r(r,z)e_r+u_\theta(r,z)e_\theta+u_z(r,z)e_z,
\qquad
\omega=\omega_r e_r+\omega_\theta e_\theta+\omega_z e_z,
\]
where $\omega_r=-\partial_z u_\theta$, $\omega_\theta=\partial_z u_r-\partial_r u_z$ and $\omega_z= r^{-1}\partial_r(ru_\theta)$. For a scalar cylindrical component $f$, we write
\[
\nabla_{r,z}f=(\partial_rf,\partial_zf),
\qquad
\Delta_{r,z}f=\partial_r^2f+\partial_z^2f.
\]
For notational simplicity, $|\nabla f|$ denotes $|\nabla_{r,z}f|$ whenever $f$ is a cylindrical scalar component. At the axis, the smoothness of an azimuthal component $f$ is understood through the associated vector field $f e_\theta$.

For $R>0$ and $z_0\in\R$, define the nested cylinders
\begin{align*}
	\Q_R(z_0)&:=\left\{x=(x',z):\frac R2<|x'|<\frac{3R}{2},\ |z-z_0|<R\right\},\\
	\Q_R'(z_0)&:=\left\{x=(x',z):\frac{3R}{4}<|x'|<\frac{5R}{4},\ |z-z_0|<\frac R2\right\},\alabel{eq:nested-cylinders}\\
	\Q_R''(z_0)&:=\left\{x=(x',z):\frac{7R}{8}<|x'|<\frac{9R}{8},\ |z-z_0|<\frac R4\right\}.
\end{align*}
Thus $\overline{\Q_R''(z_0)}\subset\Q_R'(z_0)$ and $\overline{\Q_R'(z_0)}\subset\Q_R(z_0)$, with separation distances comparable with $R$. Under the centered scaling $\widetilde x'=x'/R$ and $\widetilde z=(z-z_0)/R$, these cylinders become fixed cylinders $\widetilde\Q_1$, $\widetilde\Q_2$, and $\widetilde\Q_3$. Their meridional sections are
\begin{equation}
	\widetilde D_1:=\left(\frac12,\frac32\right)\times(-1,1),
	\qquad
	\widetilde D_2:=\left(\frac34,\frac54\right)\times\left(-\frac12,\frac12\right),
	\qquad
	\widetilde D_3:=\left(\frac78,\frac98\right)\times\left(-\frac14,\frac14\right),
	\label{eq:fixed-meridional-domains}
\end{equation}
and $\widetilde\Q_j:=\{(\widetilde x',\widetilde z):( |\widetilde x'|,\widetilde z)\in\widetilde D_j\}$ for $j=1,2,3$. In particular, $\widetilde D_3\Subset\widetilde D_2\Subset\widetilde D_1$.

Throughout the paper, $C$ denotes a positive constant whose value may change from line to line. Any relevant dependence will be indicated by subscripts or arguments, e.g., $C_\alpha$ or $C(\alpha,\beta)$. We write $X\lesssim Y$ if $X\leq CY$, and $X\sim Y$ if both $X\lesssim Y$ and $Y\lesssim X$.

The paper is organized as follows. Section \ref{sec:preliminaries} collects the main tools used in the paper. Section \ref{sec:axisymmetric-refinement} proves the new Calderón--Zygmund and annular vorticity estimates, and then proves Theorem \ref{thm:main}. Section \ref{sec:liouville-endpoint} establishes the pressure-free Caccioppoli inequality and the logarithmic vorticity-to-velocity decay transfer, and then proves Theorem \ref{thm:vorticity}. Some auxiliary estimates are placed in Appendix \ref{sec:appendix}.

	\section{Preliminaries and key estimates}
	\label{sec:preliminaries}

	\subsection{Known axisymmetric decay estimates}
	
	From \eqref{eq:NS}, the axisymmetric steady equations are
	\begin{equation*}
		\left\{
		\begin{aligned}
			(u_r\partial_r+u_z\partial_z)u_r-\frac{u_\theta^2}{r}+\partial_rp&=\left(\partial_r^2+\frac1r\partial_r+\partial_z^2-\frac1{r^2}\right)u_r,\\
			(u_r\partial_r+u_z\partial_z)u_\theta+\frac{u_ru_\theta}{r}&=\left(\partial_r^2+\frac1r\partial_r+\partial_z^2-\frac1{r^2}\right)u_\theta,\\
			(u_r\partial_r+u_z\partial_z)u_z+\partial_zp&=\left(\partial_r^2+\frac1r\partial_r+\partial_z^2\right)u_z,\\
			\partial_ru_r+\frac{u_r}{r}+\partial_zu_z&=0.
		\end{aligned}
		\right.
	\end{equation*}
	The corresponding vorticity equations are
	\begin{equation}
		\left\{
		\begin{aligned}
			(u_r\partial_r+u_z\partial_z)\omega_r-(\omega_r\partial_r+\omega_z\partial_z)u_r&=\left(\partial_r^2+\frac1r\partial_r+\partial_z^2-\frac1{r^2}\right)\omega_r,\\
			(u_r\partial_r+u_z\partial_z)\omega_\theta-\frac{u_r}{r}\omega_\theta-\frac1r\partial_z(u_\theta^2)&=\left(\partial_r^2+\frac1r\partial_r+\partial_z^2-\frac1{r^2}\right)\omega_\theta,\\
			(u_r\partial_r+u_z\partial_z)\omega_z-(\omega_r\partial_r+\omega_z\partial_z)u_z&=\left(\partial_r^2+\frac1r\partial_r+\partial_z^2\right)\omega_z.
		\end{aligned}
		\right.
		\label{eq:vorticity-equations}
	\end{equation}
	
	We will use the following uniform estimate. Fix once and for all an exponent $q_*>3$.
	
	\begin{Lem}
		\label{lem:uniform-regularity}
		Let $(u,p)$ be a smooth solution of \eqref{eq:NS}, and set $U:=\|u\|_{L^\infty(\R^3)}$. Then $U<\infty$. Moreover, there exists a constant $C_{\mathrm{reg}}=C(q_*,U)<\infty$, such that
		\begin{equation*}
			\sup_{x_0\in\R^3}\left(\|u\|_{W^{4,q_*}(B_1(x_0))}+\|\nabla p\|_{W^{2,q_*}(B_1(x_0))}\right)\leq C_{\mathrm{reg}},
		\end{equation*}
		and
		\begin{equation*}
			\sup_{x\in\R^3}\left(\sum_{j=0}^{3}|\nabla^j u(x)|+\sum_{j=1}^{2}|\nabla^j p(x)| + \sum_{j=0}^{2}|\nabla^j \omega(x)|\right)\leq C_{\mathrm{reg}}.
		\end{equation*}
	\end{Lem}
	
	The proof of this lemma relies only on standard interior estimates for the stationary Stokes system; see \cite[Chapter~IV]{G2011}.
	We will also use the following estimates from Carrillo--Pan--Zhang \cite{CPZ2020}.
	
	\begin{Prop}[see Theorem 1.1 and formula (3.18) in \cite{CPZ2020}] \label{prop:CPZ}
		Under the assumptions of Theorem \ref{thm:main}, there exist $R_*=R_*(M)\geq2$ and $A_*=A_*(M)>0$ such that, uniformly in $z\in\R$ and for $r\geq R_*$,
		\begin{equation}
			|u(r,z)|\leq A_*r^{-1/2}[\Lr r]^{1/2},
			\label{eq:known-u}
		\end{equation}
		\begin{equation*}
			|\omega_\theta(r,z)|+|\nabla\omega_\theta(r,z)|\leq A_*r^{-5/4}[\Lr r]^{3/4},
		\end{equation*}
		and
		\begin{equation*}
			|\omega_r(r,z)|+|\omega_z(r,z)|\leq A_* r^{-1}\Lr r.
		\end{equation*}
		Moreover,
		\begin{equation}
			B_*:=\sup_{0\leq r\leq R_*,\ z\in\R}|\omega_\theta(r,z)|<\infty,
			\label{eq:wtheta-core-bound}
		\end{equation}
	\end{Prop}

	We next turn to the representation formulas that connect vorticity to velocity.

	\subsection{Representation formulas and singular-integral tools}

	In the following, we will use matrix-valued kernels $K\in C^1(\R^3\setminus\{0\})$ satisfying
	\begin{equation}
		|K(x)|\leq C_K|x|^{-3},
		\qquad x\neq0,
		\label{eq:CZ-size}
	\end{equation}
	and
	\begin{equation}
		\int_{\rho_1<|x|<\rho_2}K(x)\,dx=0
		\label{eq:CZ-cancellation}
	\end{equation}
	for every $0<\rho_1<\rho_2<\infty$. Every homogeneous Calder\'on--Zygmund kernel of degree$-3$ with zero spherical average has these properties. The arguments below use only \eqref{eq:CZ-size} and \eqref{eq:CZ-cancellation}.

	\begin{Lem}
		\label{lem:Biot-Savart}
		The following two forms of the Biot--Savart representation hold.

		\textup{(i)} Let $u$ be a smooth $D$-solution of \eqref{eq:NS}, and let $\omega=\curl u$. Then, in the sense of distributions,
		\begin{equation}
			u(x)=\frac1{4\pi}\int_{\R^3}\frac{\omega(y)\times(x-y)}{|x-y|^3}\,dy.
			\label{eq:Biot-Savart}
		\end{equation}
		Whenever the integral is absolutely convergent, it agrees with the smooth representative of $u$ and satisfies
		\begin{equation}
			|u(x)|\leq C\int_{\R^3}\frac{|\omega(y)|}{|x-y|^2}\,dy.
			\label{eq:Biot-Savart-bound}
		\end{equation}

		\textup{(ii)} Assume that $b\in C^2(\R^3;\R^3)$ is bounded, $b(x)\to0$ as $|x|\to\infty$, $\diver b=0$, and $F=\curl b\in C^1(\R^3;\R^3)$ is bounded. Suppose that, for some $1<a<2$, $\beta\geq0$, $R_*\geq2$, and $A>0$,
		\begin{equation}
			|F(x',z)|\leq A|x'|^{-a}[\log(\ee+|x'|)]^\beta,
			\qquad |x'|\geq R_*,\ z\in\R.
			\label{eq:F-anisotropic-decay}
		\end{equation}
		Set
		\[
		B_F:=\sup_{|x'|\leq R_*,\ z\in\R}|F(x',z)|<\infty.
		\]
		Then
		\begin{equation}
			b(x)=\frac1{4\pi}\int_{\R^3}\frac{F(y)\times(x-y)}{|x-y|^3}\,dy,
			\label{eq:Biot-Savart-b}
		\end{equation}
		the integral is absolutely convergent, and
		\begin{equation}
			|b(x',z)|\leq CA(1+|x'|)^{1-a}[\log(\ee+|x'|)]^\beta+CB_F\frac{R_*^2}{R_*+|x'|}.
			\label{eq:Biot-Savart-b-bound}
		\end{equation}
	\end{Lem}

	\begin{proof}
		For part \textup{(i)}, since $\omega\in L^2(\R^3)$, the vector field $v:=\curl(-\Delta)^{-1}\omega$, initially defined in the distributional sense, belongs to $L^6(\R^3)$ by the Sobolev inequality. Moreover, $\diver v=0$ and $\curl v=\omega$. Since $u\in L^6(\R^3)$ by \eqref{eq:L6}, the difference $h=u-v$ belongs to $L^6(\R^3)$ and satisfies $\diver h=0$ and $\curl h=0$. Hence,
		\[
		-\Delta h=\curl(\curl h)-\nabla(\diver h)=0.
		\]
		Since an $L^6$ harmonic vector field on $\R^3$ is zero, we have $u=v$. This gives \eqref{eq:Biot-Savart} in distributions. Under absolute convergence, the distributional formula agrees pointwise with the smooth representative, and \eqref{eq:Biot-Savart-bound} follows immediately.

		For part \textup{(ii)}, write $y=(y',s)$ and integrate first in $s$. The local singularity is integrable in three dimensions, and
		\[
		\left|\frac1{4\pi}\int_{\R^3}\frac{F(y)\times(x-y)}{|x-y|^3}\,dy\right|
		\leq C\int_{\R^2}\frac{\sup_{s\in\R}|F(y',s)|}{|x'-y'|}\,dy'.
		\]
		The contribution of $|y'|\leq R_*$ is bounded by $CB_FR_*^2/(R_*+|x'|)$, while \eqref{eq:F-anisotropic-decay} and Lemma \ref{lem:potential-estimate} control the exterior contribution by $CA(1+|x'|)^{1-a}[\log(\ee+|x'|)]^\beta$. This proves absolute convergence and \eqref{eq:Biot-Savart-b-bound}. Denote the resulting vector field by $\mathcal BF$. Since $F=\curl b$, one has $\diver F=0$, and the standard distributional identities give $\diver\mathcal BF=0$ and $\curl\mathcal BF=F$. Hence $b-\mathcal BF$ is a bounded harmonic vector field and therefore a constant. Hence \eqref{eq:Biot-Savart-b} holds.
	\end{proof}

	Part \textup{(ii)} isolates the absolutely convergent representation needed in the axisymmetric argument. Differentiating that formula yields the derivative representation used below.

	\begin{Lem}
		\label{lem:differentiated-Biot-Savart}
		Under the assumptions of Lemma \ref{lem:Biot-Savart} \textup{(ii)}, for every $i,k\in\{1,2,3\}$ there exist constants $A_{ijk}$ and homogeneous Calder\'on--Zygmund kernels $K_{ijk}$ of degree$-3$ such that,
		\begin{equation}
			\partial_kb_i(x)=A_{ijk}F_j(x)+\pv\int_{\R^3}K_{ijk}(x-y)F_j(y)\,dy.
			\label{eq:differentiated-BS}
		\end{equation}
		Moreover, each $K_{ijk}$ satisfies \eqref{eq:CZ-size} and \eqref{eq:CZ-cancellation}.
	\end{Lem}
	
	This standard argument may be found, for example, in \cite[Section~1.4]{BV2022} and \cite[Chapter~II]{S1993}.

	\subsection{Local analytic estimates}
	
	We first recall the scale-invariant Bogovskii estimate on $A_R$.
	
	\begin{Lem}[see \cite{B1980} or Chapter~III, Theorem~III.3.1 in \cite{G2011}]
		\label{lem:Bogovskii}
		For every \(R>0\), there exists a linear operator \(\mathcal B_R\), whose choice is independent of \(q\), such that, for every \(1<q<\infty\),
		\[
		\mathcal B_R:L^q_0(A_R)\longrightarrow W^{1,q}_0(A_R;\mathbb R^3),
		\qquad
		L^q_0(A_R):=\left\{g\in L^q(A_R):\int_{A_R}g\,dx=0\right\}.
		\]
		Moreover, for every \(g\in L^q_0(A_R)\),
		\[
		\operatorname{div}\mathcal B_R g=g
		\quad\text{in }A_R,
		\]
		and
		\[
		\|\nabla\mathcal B_R g\|_{L^q(A_R)}
		\leq C_q\|g\|_{L^q(A_R)}.
		\]
	\end{Lem}

	The second tool is the localized Brezis--Gallouet inequality on the fixed nested domains in \eqref{eq:fixed-meridional-domains}; see \cite{BG1980} and \cite[Section~3]{CPZ2020}.

	\begin{Lem}
		\label{lem:localized-BG}
		Let $V=(V_1,V_2)\in H^2(\widetilde D_2)$. Then
		\begin{equation*}
			\|V\|_{L^\infty(\widetilde D_3)}
			\leq
			C\left(1+\|V\|_{H^1(\widetilde D_2)}\right)
			\left[\log\left(\ee+\|\widetilde\Delta V\|_{L^2(\widetilde D_2)}\right)\right]^{1/2},
		\end{equation*}
		where $\widetilde\Delta=\partial_{\widetilde r}^2+\partial_{\widetilde z}^2$, and $C$ depends only on the fixed pair $\widetilde D_3\Subset\widetilde D_2$.
	\end{Lem}

	The displayed form follows from the usual ratio form of the two-dimensional Brezis--Gallouet inequality after localization by a cut-off supported in $\widetilde D_2$ and equal to one on $\widetilde D_3$.

	\section{New Calder\'on--Zygmund estimates for axisymmetric solutions}
	\label{sec:axisymmetric-refinement}
	
	This section contains the two estimates responsible for the improved logarithmic powers in Theorem \ref{thm:main}. We first combine rotational localization with shellwise kernel cancellation, and then derive the annular $L^2$ control of $\nabla\omega_\theta$ needed in the new Calder\'on--Zygmund estimate.
	
	\subsection{A new Calder\'on--Zygmund estimate}
	
	We begin with the geometric localization consequence of axisymmetry.
	
	\begin{Lem}
		\label{lem:rotation}
		Let $g=g(r,z)$ be an axisymmetric scalar function. Fix $x_0=(R,0,z_0)$ and let $0<s\leq R/8$. Then
		\begin{equation}
			\int_{B_{2s}(x_0)}|g(x)|^2\,dx\leq C\frac{s}{R}\int_{\Q_R'(z_0)}|g(x)|^2\,dx,
			\label{eq:rotation-localization}
		\end{equation}
		where $C$ is an absolute constant.
	\end{Lem}
	
	\begin{proof}
		Choose a sufficiently small absolute constant $c_0>0$ and denote
		\[
		N:=\max\left\{1,\left\lfloor c_0\frac{R}{s}\right\rfloor\right\}.
		\]
		If $N\geq2$, choose $x_k=(R\cos\theta_k,R\sin\theta_k,z_0)$ with $\theta_k=2\pi k/N$, $k=1,\dots,N$. For $j\neq k$,
		\[
		|x_j-x_k|=2R\left|\sin\frac{\theta_j-\theta_k}{2}\right|.
		\]
		For adjacent centers, $|x_{k+1}-x_k|=2R\sin(\pi/N)\geq4R/N$ when $N\geq2$. Since $N\leq c_0R/s$, choosing $c_0\leq1/2$ we have $|x_{k+1}-x_k|\geq8s$. Hence, the balls $B_{2s}(x_k),\,k=1,\dots,N$ are pairwise disjoint. Due to $2s\leq R/4$, all of these balls lie in $\Q_R'(z_0)$. Axisymmetry and rotational invariance of Lebesgue measure imply
		\[
		\int_{B_{2s}(x_k)}|g|^2\,dx=\int_{B_{2s}(x_0)}|g|^2\,dx,
		\]
		which implies
		\[
		N\int_{B_{2s}(x_0)}|g|^2\,dx\leq\int_{\Q_R'(z_0)}|g|^2\,dx.
		\]
		When $c_0R/s\geq2$, one has $N\geq c_0R/s-1 \geq c_0R/(2s)$ and \eqref{eq:rotation-localization} follows. When $c_0R/s<2$, the ratio $s/R$ is bounded by $c_0/2$, and $B_{2s}(x_0)\subset\Q_R'(z_0)$ gives \eqref{eq:rotation-localization} after enlarging $C$.
	\end{proof}
	
	
	The packing estimate is now combined with the kernel assumptions introduced in \eqref{eq:CZ-size}--\eqref{eq:CZ-cancellation}. Let $K$ be a matrix-valued convolution kernel satisfying those assumptions. For an axisymmetric scalar function $f=f(r,z)$, define $F(x)=f(r,z)e_\theta(x)$. The following estimate is the main singular-integral input for the meridional velocity-gradient bound.
	
	\begin{Lem}
		\label{lem:refined-CZ}
		Let $0<a<2$, $\beta\geq0$, and $R_*\geq2$. Suppose that $f=f(r,z)$ belongs to $C^1([0,\infty)\times\R)$ and that the azimuthal vector field $F=f e_\theta$, initially defined for $r>0$, extends to a $C^1$ vector field on $\R^3$. Assume that
		\begin{equation}
			|f(r,z)|+|\nabla f(r,z)|\leq A r^{-a}[\Lr r]^{\beta},\qquad r\geq R_*,\ z\in\R,
			\label{eq:f-pointwise}
		\end{equation}
		and
		\begin{equation}
			B_*:=\sup_{0\leq r\leq R_*,\ z\in\R}|f(r,z)|<\infty.
			\label{eq:f-core-bound}
		\end{equation}
		Fix $x_0=(R,0,z_0)$ with $R\geq\max\{32,2R_*\}$ and define
		\begin{equation}
			\mathcal{G}_R(z_0):=\norm{\nabla f}_{L^2(\Q_R'(z_0))}.
			\label{eq:GR-definition}
		\end{equation}
		Then,
		\begin{equation*}
			\left|\pv\int_{\R^3}K(x_0-y)F(y)\,dy\right|\leq C A R^{-a}[\Lr R]^{\beta}+CB_*R_*^2R^{-2}+C R^{-1/2}\Lr R\,\mathcal{G}_R(z_0),
		\end{equation*}
		where $C$ depends only on $a$, $\beta$ and the kernel constants $C_K$.
	\end{Lem}
	
	\begin{proof}
		Write the integral as $I=I_{\mathrm{near}}+I_{\mathrm{mid}}+I_{\mathrm{far}}$, corresponding respectively to $|x_0-y|\leq1$, $1<|x_0-y|<R/8$, and $|x_0-y|\geq R/8$.
		

        \noindent\underline{\textbf{Estimate of \(I_{\mathrm{near}}\).}} Note \(R\geq R_*+1\). For \(|x_0-y|\leq1\), one has $R-1\leq r(y)\leq R+1$, and hence \(r(y)\geq R_*\) and \(r(y)\sim R\). By \eqref{eq:CZ-cancellation}, we have
        \[
        \begin{aligned}
        	I_{\mathrm{near}}
        	=\lim_{\varepsilon\to0^+}\int_{\varepsilon<|x_0-y|\leq1}K(x_0-y)F(y)\,dy=\lim_{\varepsilon\to0^+}\int_{\varepsilon<|x_0-y|\leq1}K(x_0-y)\bigl(F(y)-F(x_0)\bigr)\,dy.
        \end{aligned}
        \]
        For \(t\in[0,1]\), let $\xi_t=x_0+t(y-x_0)$. Then \(|\xi_t-x_0|\leq1\), and therefore $R-1\leq r(\xi_t)\leq R+1$ for $r(\xi_t)\geq R_*$. Using \eqref{eq:f-pointwise} and
        \[
        |\nabla F(\xi_t)|\leq C\left(|\nabla_{r,z}f(\xi_t)|+\frac{|f(\xi_t)|}{r(\xi_t)}\right),
        \]
        we obtain
        \[
        |\nabla F(\xi_t)|\leq CA R^{-a}[\Lr R]^\beta
        \]
        uniformly for \(t\in[0,1]\). Hence, we obtain
        \[
        \begin{aligned}
        	|F(y)-F(x_0)|
        	\leq |y-x_0|\int_0^1|\nabla F(\xi_t)|\,dt\leq CA R^{-a}[\Lr R]^\beta|x_0-y|,
        \end{aligned}
        \]
        which implies
        \begin{equation}
        	|I_{\mathrm{near}}|
        	\leq CA R^{-a}[\Lr R]^\beta\int_{|h|\leq1}|h|^{-2}\,dh
        	\leq CA R^{-a}[\Lr R]^\beta.
        	\label{eq:near-estimate}
        \end{equation} 
		
		\noindent\underline{\textbf{Estimate of \(I_{\mathrm{mid}}\).}} Choose $J\geq0$ such that $2^J<R/8\leq2^{J+1}$. For $0\leq j\leq J-1$, let $s_j=2^j$ and $E_j=\{s_j<|x_0-y|\leq2s_j\}$; let $s_J=2^J$ and $E_J=\{s_J<|x_0-y|<R/8\}$. Then $E_j\subset B_{2s_j}(x_0)\setminus B_{s_j}(x_0)$ and $s_j\leq R/8$. Denote $m_j:={|B_{2s_j}|^{-1}}\int_{B_{2s_j}(x_0)}f(y)\,dy$. By \eqref{eq:CZ-cancellation}, we have
		\[
		\int_{E_j}K(x_0-y)m_je_\theta(x_0)\,dy=0,
		\]
		which implies
		\begin{align*}
			\int_{E_j}K(x_0-y)F(y)\,dy
			={}&\int_{E_j}K(x_0-y)(f(y)-m_j)e_\theta(x_0)\,dy\notag\\
			&+\int_{E_j}K(x_0-y)f(y)(e_\theta(y)-e_\theta(x_0))\,dy\notag\\
			=:{}&I_j^{(1)}+I_j^{(2)}.
		\end{align*}
		For $I_j^{(1)}$, by H\"older's inequality we have
		\[
		|I_j^{(1)}|\leq\norm{K(x_0-\cdot)}_{L^2(E_j)}\norm{f-m_j}_{L^2(E_j)} \leq CR^{-1/2}\mathcal{G}_R(z_0), 
		\]
		where we used that the kernel estimate gives
		\begin{equation*}
			\norm{K(x_0-\cdot)}_{L^2(E_j)}\leq Cs_j^{-3/2},
		\end{equation*}
		and that the Poincar\'e inequality and Lemma \ref{lem:rotation} give
		\begin{equation*}
			\norm{f-m_j}_{L^2(E_j)}\leq \norm{f-m_j}_{L^2(B_{2s_j}(x_0))} \leq Cs_j\norm{\nabla_{r,z} f}_{L^2(B_{2s_j}(x_0))}\leq Cs_j\left(\frac{s_j}{R}\right)^{1/2}\mathcal{G}_R(z_0).
		\end{equation*}
		Therefore, from $J+1 \leq C\Lr R$ we infer that
		\begin{equation}
			\sum_{j=0}^J|I_j^{(1)}|\leq CR^{-1/2}\Lr R\,\mathcal{G}_R(z_0).
			\label{eq:sum-Ij1}
		\end{equation}
		
		For $y\in E_j$, $j=0,\cdots J$, one has $ R_* \leq {7}R/8 \leq r \leq {9}R/8$. Thus, writing $y'=(r\cos\varphi,r\sin\varphi)$ and $x_0'=(R,0)$, we have 
		\begin{align} \label{eq:e_theta}
			|e_\theta(y)-e_\theta(x_0)| = |(-\sin \varphi, \cos \varphi) - (0,1)|=2 |\sin \frac\varphi2 |\leq\frac{|x_0'-y'|}{\sqrt{rR}}\leq C\frac{|x_0-y|}{R}\leq C\frac{s_j}{R},
		\end{align}
		where we used that
		\begin{align*}
			|x_0'-y'|^2 = R^2 + r^2 - 2rR\cos \varphi = (R-r)^2+4 R r \sin ^2 \frac{\varphi}{2} \geq 4 R r \sin ^2 \frac{\varphi}{2}.
		\end{align*}
		It follows from \eqref{eq:CZ-size}, \eqref{eq:f-pointwise} and \eqref{eq:e_theta} that
		\[
		|I_j^{(2)}|\leq CA R^{-a}[\Lr R]^{\beta}\frac{s_j}{R}\int_{E_j}|x_0-y|^{-3}\,dy\leq CA R^{-a}[\Lr R]^{\beta}\frac{s_j}{R}.
		\]
		Hence, by $\sum_{j=0}^Js_j = 2^{J+1} - 1\leq CR$ we have
		\begin{equation}
			\sum_{j=0}^J|I_j^{(2)}|\leq CA R^{-a}[\Lr R]^{\beta}.
			\label{eq:sum-Ij2}
		\end{equation}
		Combining \eqref{eq:sum-Ij1} and \eqref{eq:sum-Ij2}, we have 
		\begin{equation}
			|I_{\mathrm{mid}}|\leq CA R^{-a}[\Lr R]^{\beta}+CR^{-1/2}\Lr R\,\mathcal{G}_R(z_0).
			\label{eq:mid-estimate}
		\end{equation}
		
		\noindent\underline{\textbf{Estimate of \(I_{\mathrm{far}}\).}} For the far field, write $y=(y',k)$ and $\rho=|y'|$, and split it into
		\[
		\Omega_1=\{\rho\geq2R\},\qquad \Omega_2=\{\rho\leq R/2\},\qquad \Omega_3=\{R/2<\rho<2R,\ |x_0-y|\geq R/8\}.
		\]
		On \(\Omega_1\), we have $|x_0'-y'|\ge \rho-R\ge {\rho}/{2}$. By \eqref{eq:CZ-size}, \eqref{eq:f-pointwise}, and cylindrical coordinates, we obtain
		\begin{align}
			\int_{\Omega_1}|K(x_0-y)||F(y)|dy
			&\le CA\int_{2R}^\infty\int_0^{2\pi}
			\rho^{-a}[\Lr\rho]^\beta
			\frac{\rho}{\rho^2}d\varphi d\rho \nonumber\\
			&\le CA\int_{2R}^\infty
			\rho^{-a-1}[\Lr\rho]^\beta d\rho \nonumber\\
			&\le CA R^{-a}[\Lr R]^\beta, \label{eq:omega1}
		\end{align}
		where we used that
		\begin{align} \label{eq:2d-Ker}
			\int_{\mathbb R}\frac{dk}{(|x_0'-y'|^2+(z_0-k)^2)^{3/2}}
			=\frac{2}{|x_0'-y'|^2}.
		\end{align}
		
		On $\Omega_2$, $|x_0'-y'|\geq R/2$. Splitting the radial integral at $R_*$ and using \eqref{eq:f-core-bound} on the inner part and \eqref{eq:f-pointwise} on the outer part, we obtain
		\begin{align}
			\int_{\Omega_2}|K(x_0-y)||F(y)|\,dy
			&\leq CB_*R^{-2}\int_0^{R_*}\rho\,d\rho+CAR^{-2}\int_{R_*}^{R/2}\rho^{1-a}[\Lr\rho]^\beta\,d\rho \nonumber\\
			&\leq CB_*R_*^2R^{-2}+CA R^{-a}[\Lr R]^\beta, \label{eq:omega2}
		\end{align}
		where \eqref{eq:2d-Ker} and $a<2$ are used. 
		
		Finally, on $\Omega_3$, one scales $y=x_0+R\widetilde y$ and has
		\begin{align*}
			\int_{\Omega_3}|x_0-y|^{-3}\,dy =\int_{1/2<|e_1+\widetilde y'|<2,\,|\widetilde y|\geq1/8}|\widetilde y|^{-3}\,d\widetilde y \leq \int_{|\widetilde y'|<3,\,|\widetilde y|\geq1/8}|\widetilde y|^{-3}\,d\widetilde y.
		\end{align*}
			We split the last integral into the regions
		$|\widetilde y_3|\leq 1$ and $|\widetilde y_3|>1$.
		On the first region, the condition $|\widetilde y|\geq 1/8$ removes the
		singularity at the origin, and therefore
		\[
		\int_{|\widetilde y'|<3,\, |\widetilde y_3|\leq1,\,
				|\widetilde y|\geq1/8}
		|\widetilde y|^{-3}\,d\widetilde y
		\leq C.
		\]
		On the second region, by $|\widetilde y|\geq |\widetilde y_3|$ we have
		\[
		\int_{|\widetilde y'|<3,\, |\widetilde y_3|>1}
		|\widetilde y|^{-3}\,d\widetilde y
		\le
		C\int_{|t|>1}|t|^{-3}\,dt
		\leq C.
		\]
		Thus, we have
		\[
		\int_{\Omega_3}|x_0-y|^{-3}\,dy\leq C,
		\]
		which, together with \eqref{eq:f-pointwise}, implies
		\begin{align} \label{eq:omega3}
			\int_{\Omega_3}|K(x_0-y)||F(y)|\,dy\leq CA R^{-a}[\Lr R]^{\beta}.
		\end{align}
		Combining \eqref{eq:omega1}, \eqref{eq:omega2}, and \eqref{eq:omega3}, we obtain
		\begin{equation}
			|I_{\mathrm{far}}|\leq CA R^{-a}[\Lr R]^{\beta}+CB_*R_*^2R^{-2}.
			\label{eq:far-estimate}
		\end{equation}
		 
		 Adding \eqref{eq:near-estimate}, \eqref{eq:mid-estimate}, and \eqref{eq:far-estimate} together, we complete the proof.
	\end{proof}
	
	\subsection{Annular estimates for the azimuthal vorticity}
	
	We next estimate the quantity $\mathcal G_R(z_0)$ appearing in Lemma \ref{lem:refined-CZ}. Recall that $\Q_R(z_0)$ and $\Q_R'(z_0)$ are the nested cylinders defined in \eqref{eq:nested-cylinders}.
	
	\begin{Lem}
		\label{lem:wtheta-energy}
		Let $R\geq2$, $z_0\in\R$, and $U_R(z_0):=\norm{u}_{L^\infty(\Q_R(z_0))}$. Then
		\begin{equation}
			\norm{\nabla\omega_\theta}_{L^2(\Q_R'(z_0))}^2\leq CR^{-2}\bigl(1+RU_R(z_0)\bigr)\norm{(\omega_r,\omega_\theta)}_{L^2(\Q_R(z_0))}^2.
			\label{eq:scaled-wtheta-energy}
		\end{equation}
	\end{Lem}
	
	\begin{proof}
		Translate $z_0$ to zero and use the scaling
		\[
		\widetilde x=\frac{x}{R},\qquad \widetilde u(\widetilde x)=Ru(R\widetilde x),\qquad \widetilde\omega(\widetilde x)=R^2\omega(R\widetilde x).
		\]
		The cylinders $\Q_R(z_0)$ and $\Q_R'(z_0)$ then become the fixed cylinders $\widetilde\Q_1$ and $\widetilde\Q_2$ defined after \eqref{eq:fixed-meridional-domains}.
		Choose an axisymmetric cut-off $\psi\in C_c^\infty(\widetilde\Q_1)$ such that $\psi=1$ on $\widetilde\Q_2$ and $|\nabla\psi|\leq C$. 
		
		For simplicity, we drop tildes during the following fixed-scale calculation. Multiplying the $\omega_\theta$ equation in \eqref{eq:vorticity-equations} by $-\omega_\theta\psi^2$ and integrating in $\widetilde\Q_1$ give
		\begin{align}
			\int |\nabla(\omega_\theta\psi)|^2\,dx+\int\frac{\omega_\theta^2\psi^2}{r^2}\,dx
			&=\int\omega_\theta^2|\nabla\psi|^2\,dx+\frac12\int\omega_\theta^2(u_r\partial_r+u_z\partial_z)(\psi^2)\,dx\notag\\
			&\quad+\int\frac{u_r}{r}\omega_\theta^2\psi^2\,dx-2\int\frac{u_\theta\omega_r\omega_\theta}{r}\psi^2\,dx.
			\label{eq:wtheta-energy-identity}
		\end{align}
		Here, the transport term has been integrated by parts using $\diver b=0$ and we have used $r^{-1}\partial_z(u_\theta^2)=-2r^{-1}u_\theta\omega_r$. Since $1/2<r<3/2$ on the support of $\psi$ and $|\nabla\psi|\leq C$, the four terms on the right-hand side of \eqref{eq:wtheta-energy-identity} satisfy
		\[
		\int\omega_\theta^2|\nabla\psi|^2\,dx\leq C\|\omega_\theta\|_{L^2(\widetilde\Q_1)}^2,
		\]
		\[
		\left|\frac12\int\omega_\theta^2b\cdot\nabla(\psi^2)\,dx\right|\leq C\|u\|_{L^\infty(\widetilde\Q_1)}\|\omega_\theta\|_{L^2(\widetilde\Q_1)}^2,
		\]
		\[
		\left|\int\frac{u_r}{r}\omega_\theta^2\psi^2\,dx\right|\leq C\|u\|_{L^\infty(\widetilde\Q_1)}\|\omega_\theta\|_{L^2(\widetilde\Q_1)}^2,
		\]
		and
		\[
		\left|2\int\frac{u_\theta\omega_r\omega_\theta}{r}\psi^2\,dx\right|\leq C\|u\|_{L^\infty(\widetilde\Q_1)}\left(\|\omega_r\|_{L^2(\widetilde\Q_1)}^2+\|\omega_\theta\|_{L^2(\widetilde\Q_1)}^2\right).
		\]
		Consequently, the left-hand side of \eqref{eq:wtheta-energy-identity} is bounded by $C(1+\|u\|_{L^\infty(\widetilde\Q_1)})\|(\omega_r,\omega_\theta)\|_{L^2(\widetilde\Q_1)}^2$. This implies
		\begin{equation}
			\norm{\nabla\widetilde\omega_\theta}_{L^2(\widetilde\Q_2)}^2\leq C\bigl(1+\norm{\widetilde u}_{L^\infty(\widetilde\Q_1)}\bigr)\norm{(\widetilde\omega_r,\widetilde\omega_\theta)}_{L^2(\widetilde\Q_1)}^2.
			\label{eq:fixed-wtheta-energy}
		\end{equation}
		The scaling relations are
		\[
		\norm{\nabla\widetilde\omega_\theta}_{L^2(\widetilde\Q_2)}^2=R^3\norm{\nabla\omega_\theta}_{L^2(\Q_R'(z_0))}^2,
		\]
		\[
		\norm{(\widetilde\omega_r,\widetilde\omega_\theta)}_{L^2(\widetilde\Q_1)}^2=R\norm{(\omega_r,\omega_\theta)}_{L^2(\Q_R(z_0))}^2,
		\qquad
		\norm{\widetilde u}_{L^\infty(\widetilde\Q_1)}=RU_R(z_0).
		\]
		Substituting these identities into \eqref{eq:fixed-wtheta-energy} gives \eqref{eq:scaled-wtheta-energy}.
	\end{proof}
	
	Combining this localized energy inequality with the annular hypothesis and the known velocity decay gives the quantitative bound required in Lemma \ref{lem:refined-CZ}.
	
	\begin{Cor}
		\label{cor:annular-gradient}
		Under the assumptions of Theorem \ref{thm:main}, we have
		\begin{equation*}
			\norm{\nabla\omega_\theta}_{L^2(\Q_R'(z_0))}\leq C(M) R^{-3/4}[\Lr R]^{1/4}
		\end{equation*}
		for all $R \geq 2R_*$, uniformly in $z_0\in\R$.
	\end{Cor}
	
	\begin{proof}
		Since $\Q_R(z_0) \subset \A_{R/2}\cup\A_R$ and $|\omega|\leq C|\nabla u|$, the annular assumption \eqref{eq:annular-assumption} implies that for $R\geq2$, 
		\[
		\norm{(\omega_r,\omega_\theta)}_{L^2(\Q_R(z_0))}^2\leq CM.
		\]
		By \eqref{eq:known-u}, we have for $R \geq 2R_*$,
		\[
		U_R(z_0)\leq C(M) R^{-1/2}[\Lr R]^{1/2}.
		\]
		Lemma \ref{lem:wtheta-energy} therefore yields that for $R \geq 2R_* $, 
		\[
		\norm{\nabla\omega_\theta}_{L^2(\Q_R'(z_0))}^2\leq C(M) R^{-2}\left(1+R^{1/2}[\Lr R]^{1/2}\right)\leq C(M) R^{-3/2}[\Lr R]^{1/2}.
		\]
	\end{proof}
	
	\subsection{Proof of Theorem \ref{thm:main}}
	
	\begin{proof}[Proof of Theorem \ref{thm:main}]
	Let
	\[
	b=u_re_r+u_ze_z,\qquad F=\omega_\theta e_\theta.
	\]
	Then $\diver b=0$, $\curl b=F$, and $\diver F=0$.
	
	\noindent\underline{\bf Improved decay of the meridional velocity gradient.} We first prove \eqref{eq:grad-b-final}. Note $x_0=(R,0,z_0)$. By Proposition \ref{prop:CPZ}, Lemma \ref{lem:refined-CZ} applies to $f=\omega_\theta$ with
	\[
	a=\frac54,\qquad \beta=\frac34,\qquad A=A_*(M),
	\]
	and with the fixed core radius $R_*$ and core bound $B_*$ from \eqref{eq:wtheta-core-bound}. Moreover, Corollary \ref{cor:annular-gradient} shows that the quantity $\mathcal G_R(z_0)$ defined in \eqref{eq:GR-definition} satisfies
	\begin{align}
		\mathcal G_R(z_0)\leq C(M)R^{-3/4}[\Lr R]^{1/4}, \quad R\geq 2R_*.
		\label{eq:nabla-omega-theta}
	\end{align}
	
	Since $b$ is bounded and vanishes at infinity, while $F=\curl b$ is smooth, bounded, and satisfies \eqref{eq:F-anisotropic-decay} by Proposition \ref{prop:CPZ}, all the hypotheses of Lemma \ref{lem:differentiated-Biot-Savart} are satisfied. Hence, by \eqref{eq:differentiated-BS},
	\begin{align} \label{eq:BS-nabla-b}
		\partial_kb_i(x_0)=A_{ijk}F_j(x_0)+\pv\int_{\R^3}K_{ijk}(x_0-y)F_j(y)\,dy.
	\end{align}
	Since $|F(x_0)|=|\omega_\theta(x_0)|$, we infer from Proposition \ref{prop:CPZ} that for $R\geq R_*$,
	\begin{align} \label{eq:F}
		|A_{ijk}F_j(x_0)|\leq C(M)R^{-5/4}[\Lr R]^{3/4}.
	\end{align}
	Applying Lemma \ref{lem:refined-CZ}, for $R\geq\max\{32,2R_*\}$ we have
	\begin{align}
		\left|\pv\int_{\R^3}K_{ijk}(x_0-y)F_j(y)\,dy\right| &\leq C(M)R^{-5/4}[\Lr R]^{3/4}+CB_*R_*^2R^{-2}\nonumber\\
		&\quad+CR^{-1/2}[\Lr R]\mathcal G_R(z_0). \label{eq:KF}
	\end{align}
	The last term is precisely where the improved logarithmic power enters. Indeed, by \eqref{eq:nabla-omega-theta}, we have
	\begin{align}
		R^{-1/2}[\Lr R]\mathcal G_R(z_0) &\leq C(M)R^{-1/2}[\Lr R]R^{-3/4}[\Lr R]^{1/4}\nonumber\\
		&=C(M)R^{-5/4}[\Lr R]^{5/4}. \label{eq:KF2}
	\end{align}
	Combining \eqref{eq:BS-nabla-b}--\eqref{eq:KF2}, we obtain
	\begin{equation}
		|\nabla b(x_0)|\leq C(M)R^{-5/4}[\Lr R]^{5/4}+CB_*R_*^2R^{-2}, \quad R\geq\max\{32,2R_*\}.
		\label{eq:improved-grad-b-preabsorb}
	\end{equation}
	
	Since
	\[
	R^{-2}=o\!\left(R^{-5/4}[\Lr R]^{5/4}\right),\quad\text{as }R\to\infty,
	\]
	we may choose $R_3=R_3(M,u)\geq\max\{32,2R_*\}$ sufficiently large so that the core term in \eqref{eq:improved-grad-b-preabsorb} is absorbed for every $R\geq R_3$. Consequently,
	\begin{equation}
		|\nabla b(x_0)|\leq C(M)R^{-5/4}[\Lr R]^{5/4},\qquad R\geq R_3.
		\label{eq:improved-grad-b}
	\end{equation}
	Note that
	\[
	|\partial_ru_r|+|\partial_zu_r|+|\partial_ru_z|+|\partial_zu_z|\leq C\bigl(|\partial_rb|+|\partial_zb|\bigr)\leq C|\nabla b|, 
	\]
	which, along with \eqref{eq:improved-grad-b}, gives
	\[
	|\partial_ru_r(R,z_0)|+|\partial_zu_r(R,z_0)|+|\partial_ru_z(R,z_0)|+|\partial_zu_z(R,z_0)|\leq C(M)R^{-5/4}[\Lr R]^{5/4}.
	\]
	This is \eqref{eq:grad-b-final}.
	
	\noindent\underline{\bf Improved decay of the meridional vorticity.} We next prove \eqref{eq:wm-final}. Fix $R>0$ and $z_0\in\R$, and introduce the scaling
	\[
	\widetilde x'=\frac{x'}{R},\qquad \widetilde z=\frac{z-z_0}{R},\qquad \widetilde u(\widetilde x)=Ru(x',z),\qquad \widetilde\omega(\widetilde x)=R^2\omega(x',z).
	\]
	Under this, $\Q_R(z_0)$, $\Q_R'(z_0)$, and $\Q_R''(z_0)$ become the fixed cylinders $\widetilde\Q_1$, $\widetilde\Q_2$, and $\widetilde\Q_3$ introduced after \eqref{eq:fixed-meridional-domains}. Their meridional sections satisfy $\widetilde D_3\Subset\widetilde D_2\Subset\widetilde D_1$ with fixed positive separation distances. For an axisymmetric scalar function $h=h(\widetilde r,\widetilde z)$, we have
	\begin{equation}
		\|h\|_{L^2(\widetilde\Q_j)}^2=2\pi\int_{\widetilde D_j}|h(\widetilde r,\widetilde z)|^2\widetilde r\,d\widetilde r\,d\widetilde z, \quad j=1,2,3.
		\label{eq:two-three-dimensional-norm}
	\end{equation}
	The same identity holds for meridional derivatives. 
	
	Let $\widetilde W:=(\widetilde\omega_r,\widetilde\omega_z)$ and $\widetilde U_m:=(\widetilde u_r,\widetilde u_z)$. Repeating the fixed-scale localized energy argument of \cite[formula (3.10)]{CPZ2020}, with a cut-off $\psi\in C_c^\infty(\widetilde\Q_1)$ satisfying $\psi=1$ on $\widetilde\Q_2$, gives
	\begin{equation}
		\|\widetilde\nabla\widetilde W\|_{L^2(\widetilde D_2)}^2\leq C\left(1+\|\widetilde U_m\|_{L^\infty(\widetilde\Q_1)}+\|\widetilde\nabla\widetilde U_m\|_{L^\infty(\widetilde\Q_1)}\right)\|\widetilde W\|_{L^2(\widetilde\Q_1)}^2.
		\label{eq:CPZ-gradient-energy}
	\end{equation}
	The constant depends only on the fixed separation of $\widetilde\Q_2$ from $\partial\widetilde\Q_1$. The proof follows by testing the scaled $\omega_r$ and $\omega_z$ equations against $\omega_r\psi^2$ and $\omega_z\psi^2$, integrating the transport terms by parts, and estimating the stretching and cut-off terms as in \cite[(3.5), (3.8), and (3.10)]{CPZ2020}. Since $\widetilde D_2\subset\widetilde D_1$, \eqref{eq:two-three-dimensional-norm} gives $\|\widetilde W\|_{L^2(\widetilde D_2)}\leq C\|\widetilde W\|_{L^2(\widetilde\Q_1)}$. Together with \eqref{eq:CPZ-gradient-energy}, this implies
	\begin{equation}
		\|\widetilde W\|_{H^1(\widetilde D_2)}\leq C\left(1+\|\widetilde U_m\|_{L^\infty(\widetilde\Q_1)}^{1/2}+\|\widetilde\nabla\widetilde U_m\|_{L^\infty(\widetilde\Q_1)}^{1/2}\right)\|\widetilde W\|_{L^2(\widetilde\Q_1)}.
		\label{eq:CPZ-local-H1}
	\end{equation}
	Combining Lemma \ref{lem:localized-BG} with \eqref{eq:CPZ-local-H1}, we obtain
	\begin{equation}
		\|\widetilde W\|_{L^\infty(\widetilde\Q_3)}\leq C\widetilde H_m\left[\log\left(\ee+\|\widetilde\Delta\widetilde W\|_{L^2(\widetilde\Q_2)} \right)\right]^{1/2},
		\label{eq:CPZ-fixed-scale-nested}
	\end{equation}
	where
	\begin{equation*}
		\widetilde H_m:=1+\left(1+\|\widetilde U_m\|_{L^\infty(\widetilde\Q_1)}^{1/2}+\|\widetilde\nabla\widetilde U_m\|_{L^\infty(\widetilde\Q_1)}^{1/2}\right)\|\widetilde W\|_{L^2(\widetilde\Q_1)}.
	\end{equation*}
	
	For simplicity, define
	\[
	W_m(R,z_0):=\|(\omega_r,\omega_z)\|_{L^\infty(\Q_R''(z_0))},
	\]
	\[
	U_m(R,z_0):=\|(u_r,u_z)\|_{L^\infty(\Q_R(z_0))},\quad G_m(R,z_0):=\|(\nabla_{r,z}u_r,\nabla_{r,z}u_z)\|_{L^\infty(\Q_R(z_0))},
	\]
	\[
	E_m(R,z_0):=\|(\omega_r,\omega_z)\|_{L^2(\Q_R(z_0))},\quad D_m(R,z_0):=\|(\Delta_{r,z}\omega_r,\Delta_{r,z}\omega_z)\|_{L^2(\Q_R'(z_0))}.
	\]
	The scaling relations are
	\[
	\|\widetilde W\|_{L^\infty(\widetilde\Q_3)}=R^2W_m,\quad \|\widetilde W\|_{L^2(\widetilde\Q_1)}=R^{1/2}E_m,
	\]
	\[
	\|\widetilde U_m\|_{L^\infty(\widetilde\Q_1)}=RU_m,\quad \|\widetilde\nabla\widetilde U_m\|_{L^\infty(\widetilde\Q_1)}=R^2G_m,
	\]
	and
	\[
	\|\widetilde\Delta\widetilde W\|_{L^2(\widetilde\Q_2)}=R^{5/2}D_m.
	\]
	Consequently, we infer from \eqref{eq:CPZ-fixed-scale-nested} that
	\begin{equation}
		R^2W_m(R,z_0)\leq CH_m(R,z_0)\left[\log\left(\ee+R^{5/2}D_m(R,z_0) \right)\right]^{1/2},
		\label{eq:CPZ-scaled-nested}
	\end{equation}
	where
	\begin{equation}
		H_m(R,z_0):=1+R^{1/2}\left(1+R^{1/2}U_m(R,z_0)^{1/2}+RG_m(R,z_0)^{1/2}\right)E_m(R,z_0).
		\label{eq:Hm-scaled}
	\end{equation}

	Since $\Q_R(z_0)\subset\mathcal A_{R/2}\cup\mathcal A_R$ and $|\omega|\leq C|\nabla u|$, \eqref{eq:annular-assumption} gives for $R\geq2$,
	\begin{equation}
		E_m(R,z_0)\leq C(M).
		\label{eq:Em-bound}
	\end{equation}
	For $D_m$, note 
	\[
	|\Delta_{r,z}\omega_r|+|\Delta_{r,z}\omega_z|\leq C\sum_{j=0}^2|\nabla^j\omega|
	\]
	on $\Q_R'(z_0)$. Lemma \ref{lem:uniform-regularity} and $|\Q_R'(z_0)|\leq CR^3$ then imply
	\begin{equation}
		D_m(R,z_0)\leq C C_{\mathrm{reg}}R^{3/2}.
		\label{eq:Dm-polynomial}
	\end{equation}
	If $R\geq2R_3$, then every point of $\Q_R(z_0)$ has radial coordinate comparable with $R$ and at least $\max\{R_*,R_3\}$. Hence \eqref{eq:known-u} and \eqref{eq:grad-b-final} give
	\begin{equation}
		U_m(R,z_0)\leq C(M)R^{-1/2}[\Lr R]^{1/2},\qquad G_m(R,z_0)\leq C(M)R^{-5/4}[\Lr R]^{5/4}.
		\label{eq:Um-Gm-bounds}
	\end{equation}
	Combining \eqref{eq:Hm-scaled}, \eqref{eq:Em-bound}, and \eqref{eq:Um-Gm-bounds}, we obtain
	\[
	H_m(R,z_0)\leq C(M)\left(1+R^{1/2}+R^{3/4}[\Lr R]^{1/4}+R^{7/8}[\Lr R]^{5/8}\right)\leq C(M)R.
	\]
	Set $R_D:=\max\{2,CC_{\mathrm{reg}}\}$. For $R\geq R_D$, \eqref{eq:Dm-polynomial} gives $R^{5/2}D_m(R,z_0)\leq R^5$. Consequently, for $R\geq\max\{2R_3,R_D\}$, we have
	\begin{equation}
		\left[\log\left(\ee+R^{5/2}D_m(R,z_0) \right)\right]^{1/2}\leq C[\Lr R]^{1/2}.
		\label{eq:log-factor}
	\end{equation}
	Substituting \eqref{eq:Em-bound} and \eqref{eq:log-factor} into \eqref{eq:CPZ-scaled-nested}, and using \eqref{eq:Hm-scaled}, we have
	\begin{align*}
		W_m(R,z_0) &\leq C(M)R^{-2}[\Lr R]^{1/2} \nonumber\\ &\quad +C(M)R^{-3/2}[\Lr R]^{1/2}\left(1+R^{1/2}U_m(R,z_0)^{1/2}+RG_m(R,z_0)^{1/2}\right).
	\end{align*}
	The velocity contribution satisfies
	\[
	R^{-3/2}[\Lr R]^{1/2}R^{1/2}U_m(R,z_0)^{1/2}\leq C(M)R^{-5/4}[\Lr R]^{3/4},
	\]
	whereas the gradient contribution, using \eqref{eq:grad-b-final}, satisfies
	\[
	R^{-3/2}[\Lr R]^{1/2}RG_m(R,z_0)^{1/2}\leq C(M)R^{-9/8}[\Lr R]^{9/8}.
	\]
	Therefore,
	\[
	W_m(R,z_0)
	\le
	C(M)R^{-9/8}[\Lr R]^{9/8}.
	\]
	
	Finally, choose $R_0\geq\max\{2R_3,R_D\}$. Given any point $(r,\theta,z)$ with $r\geq R_0$, take $R=r$ and $z_0=z$. Then $(r,\theta,z)\in\Q_R''(z_0)$, and therefore
	\[
	|\omega_r(r,z)|+|\omega_z(r,z)|\leq C(M)r^{-9/8}[\Lr r]^{9/8}.
	\]
	This completes the proof.
    \end{proof}
	
	\section{Logarithmic endpoint Liouville criteria}
	\label{sec:liouville-endpoint}
	
	In this section, we first obtain a radial-envelope criterion, and then transfer the assumed cylindrical vorticity decay to the critical velocity decay.
	
	\subsection{A radial-envelope criterion}
	
	Recall the spherical balls $B_R$ and annuli $A_R$ from \eqref{eq:geometric-sets}. The Bogovskii operator from Lemma \ref{lem:Bogovskii} permits a divergence-free cutoff and removes the pressure term.
	
	\begin{Lem}
		\label{lem:Caccioppoli}
		Let $u$ be a smooth solution of \eqref{eq:NS}. Then
		\begin{equation}
			\int_{B_R}|\nabla u|^2\,dx
			\leq
			C\int_{A_R}|\nabla u|^2\,dx
			+\frac{C}{R^2}\int_{A_R}|u|^2\,dx
			+\frac{C}{R}\int_{A_R}|u|^3\,dx.
			\label{eq:Caccioppoli}
		\end{equation}
	\end{Lem}
	
	\begin{proof}
		Choose $\eta_R\in C_c^\infty(B_{2R})$ such that $0\leq\eta_R\leq1$, $\eta_R=1$ on $B_R$, and $|\nabla\eta_R|\leq C/R$. Set $g_R:=u\cdot\nabla\eta_R$. Since $\diver u=0$ and $\eta_Ru$ is compactly supported, we have
		\begin{equation*}
			\int_{A_R}g_R\,dx
			=
			\int_{\R^3}\diver(\eta_Ru)\,dx
			=0.
		\end{equation*}
		Let $w_R$ be the Bogovskii correction associated with $g_R$. Since $g_R\in L^2(A_R)\cap L^3(A_R)$, we infer from Lemma \ref{lem:Bogovskii} that
		\begin{equation}
			\diver w_R=g_R,
			\qquad
			\norm{\nabla w_R}_{L^q(A_R)}\leq \frac{C_q}{R}\norm{u}_{L^q(A_R)},
			\qquad q=2,3.
			\label{eq:wR}
		\end{equation}
		Extend $w_R$ by zero outside $A_R$ (still denoted by \(w_R\)) and define $\varphi_R:=\eta_Ru-w_R$. Then $\varphi_R$ is compactly supported and divergence free. By density, it is an admissible test function for \eqref{eq:NS}; hence the pressure term vanishes and
		\begin{equation}
			\int_{\R^3}\nabla u:\nabla\varphi_R\,dx
			+
			\int_{\R^3}(u\cdot\nabla)u\cdot\varphi_R\,dx
			=0.
			\label{eq:test-identity}
		\end{equation}
		The diffusion term in \eqref{eq:test-identity} expands as
		\begin{equation}
			\int_{\R^3}\nabla u:\nabla\varphi_R\,dx
			=
			\int_{\R^3}\eta_R|\nabla u|^2\,dx
			+
			\int_{A_R}\nabla u:(u\otimes\nabla\eta_R)\,dx
			-
			\int_{A_R}\nabla u:\nabla w_R\,dx.
			\label{eq:diffusion-expansion}
		\end{equation}
		Using \eqref{eq:wR} with $q=2$, we obtain
		\begin{align}
			\left|\int_{A_R}\nabla u:(u\otimes\nabla\eta_R)\,dx\right|
			+
			\left|\int_{A_R}\nabla u:\nabla w_R\,dx\right|
			&\leq \frac{C}{R}\norm{\nabla u}_{L^2(A_R)}\norm{u}_{L^2(A_R)}\notag\\
			&\leq C\norm{\nabla u}_{L^2(A_R)}^2+\frac{C}{R^2}\norm{u}_{L^2(A_R)}^2.
			\label{eq:diffusion-bound}
		\end{align}
		For the convection term containing $\eta_Ru$, incompressibility and integration by parts give
		\begin{equation*}
			\int_{\R^3}(u\cdot\nabla)u\cdot(\eta_Ru)\,dx
			=-\frac12\int_{A_R}|u|^2u\cdot\nabla\eta_R\,dx,
		\end{equation*}
		and therefore
		\begin{equation}
			\left|\int_{\R^3}(u\cdot\nabla)u\cdot(\eta_Ru)\,dx\right|
			\leq
			\frac{C}{R}\int_{A_R}|u|^3\,dx.
			\label{eq:convection-cutoff-bound}
		\end{equation}
		Finally, using $\diver u=0$, the zero trace of $w_R$, and \eqref{eq:wR} with $q=3$, we integrate by parts to obtain
		\begin{align}
			\left|\int_{A_R}(u\cdot\nabla)u\cdot w_R\,dx\right|
			&=
			\left|\int_{A_R}u\otimes u:\nabla w_R\,dx\right|\notag\\
			&\leq
			\norm{u}_{L^3(A_R)}^2\norm{\nabla w_R}_{L^3(A_R)}
			\leq
			\frac{C}{R}\norm{u}_{L^3(A_R)}^3.
			\label{eq:convection-Bogovskii}
		\end{align}
		Substituting \eqref{eq:diffusion-expansion}, \eqref{eq:diffusion-bound}, \eqref{eq:convection-cutoff-bound}, and \eqref{eq:convection-Bogovskii} into \eqref{eq:test-identity}, and using $\eta_R=1$ on $B_R$, proves \eqref{eq:Caccioppoli}.
	\end{proof}
	
	With the pressure-free estimate established, we next prove Proposition \ref{thm:envelope} and its logarithmic endpoint consequence, Theorem \ref{thm:velocity}.
	
	\begin{proof}[Proof of Proposition \ref{thm:envelope}]
		Fix ${r_*}>1$ and split the annulus into
		\begin{equation*}
			A_R^{\mathrm{in}}({r_*}):=A_R\cap\{r\leq {r_*}\},
			\qquad
			A_R^{\mathrm{out}}({r_*}):=A_R\cap\{r>{r_*}\}.
		\end{equation*}
		For $R>2{r_*}$, every point in $A_R^{\mathrm{in}}({r_*})$ satisfies $|z|\geq\sqrt{R^2-r_*^2}$. Consequently, by the uniform vanishing of $u$ at spatial infinity, we have
		\begin{equation}
			\varepsilon_R({r_*}):=\sup_{A_R^{\mathrm{in}}({r_*})}|u|\longrightarrow0,
			\qquad\text{as }R\to\infty,
			\label{eq:epsilon}
		\end{equation}
		for every fixed ${r_*}$. Moreover, $|A_R^{\mathrm{in}}({r_*})|\leq Cr_*^2 R$. Hence, it follows from \eqref{eq:epsilon} that
		\begin{equation}
			\frac1R\int_{A_R^{\mathrm{in}}({r_*})}|u|^3\,dx
			\leq
			Cr_*^2\varepsilon_R({r_*})^3
			\longrightarrow0
			\label{eq:inner-L3}
		\end{equation}
		and
		\begin{equation}
			\frac1{R^2}\int_{A_R^{\mathrm{in}}(r_*)}|u|^2\,dx
			\leq
			\frac{Cr_*^2}{R}\varepsilon_R(r_*)^2
			\longrightarrow0.
			\label{eq:inner-L2}
		\end{equation}
		For the outer part, $|z|\leq \sqrt{4R^2-r_*^2}\leq2R$. Recalling the definition of $H$ in \eqref{eq:H}, we have
		\begin{equation}
			\frac1R\int_{A_R^{\mathrm{out}}(r_*)}|u|^3\,dx
			\leq
			C\int_{r_*}^{2R}rH(r)^3\,dr
			\leq
			C\int_{r_*}^\infty rH(r)^3\,dr.
			\label{eq:outer-L3}
		\end{equation}
		Similarly, H\"older's inequality yields
		\begin{align}
			\frac1{R^2}\int_{A_R^{\mathrm{out}}(r_*)}|u|^2\,dx
			&\leq
			\frac{C}{R}\int_{r_*}^{2R}rH(r)^2\,dr \leq
			\frac{C}{R}
			\left(\int_{r_*}^{2R}rH(r)^3\,dr\right)^{2/3}
			\left(\int_{r_*}^{2R}r\,dr\right)^{1/3}\notag\\
			&\leq
			CR^{-1/3}
			\left(\int_{r_*}^\infty rH(r)^3\,dr\right)^{2/3}
			\longrightarrow0
			\label{eq:outer-L2}
		\end{align}
		as $R\to\infty$. Applying Lemma \ref{lem:Caccioppoli}, using $\nabla u\in L^2(\R^3)$ and \eqref{eq:inner-L3}--\eqref{eq:outer-L2}, and letting $R\to\infty$, we obtain
		\begin{align*}
			\int_{\R^3}|\nabla u|^2\,dx
			\leq C\lim_{R\to\infty}
			\int_{A_R}|\nabla u|^2\,dx
			+\lim_{R\to\infty}\frac{C}{R^2}\int_{A_R}|u|^2\,dx
			+\lim_{R\to\infty} \frac{C}{R}\int_{A_R}|u|^3\,dx \leq
			C\int_{r_*}^\infty rH(r)^3\,dr.
		\end{align*}
		Hence, by letting $r_*\to\infty$ in the above inequality and using \eqref{eq:envelope-condition}, we conclude that $\int_{\R^3}|\nabla u|^2\,dx=0$. Thus $u$ is constant, and the condition $u(x)\to0$ as $|x|\to\infty$ implies $u\equiv0$.
	\end{proof}
	
	\begin{proof}[Proof of Theorem \ref{thm:velocity}]
		By \eqref{eq:velocity-endpoint}, one has
		\begin{equation*}
			\int_1^\infty rH(r)^3\,dr
			\leq
			C\int_1^\infty\frac{dr}{r[\log(\ee+r)]^{3\gamma}}.
		\end{equation*}
		The last integral is finite when $3\gamma>1$. The conclusion then follows from Proposition \ref{thm:envelope}.
	\end{proof}

	\subsection{Logarithmic decay transfer from vorticity to velocity and proof of Theorem \ref{thm:vorticity}}

	We now prove that the logarithmic factor in a cylindrical vorticity assumption is preserved when one passes from $\omega$ to $u$.

	\begin{Prop}
		\label{prop:decay-transfer}
		Let $u$ be a smooth $D$-solution of \eqref{eq:NS}, let $\omega=\curl u$, and suppose that, for some $1<\beta<2$ and $\gamma\geq0$,
		\begin{equation}
			\sup_{\substack{|x'|=r\\ z\in\R}}|\omega(x',z)|
			\leq
			Cr^{-\beta}L(r)^{-\gamma},
			\qquad r\geq1.
			\label{eq:general-vorticity-decay}
		\end{equation}
		Then
		\begin{equation}
			\sup_{\substack{|x'|=r\\ z\in\R}}|u(x',z)|
			\leq
			Cr^{1-\beta}L(r)^{-\gamma},
			\qquad r\geq4.
			\label{eq:general-velocity-decay}
		\end{equation}
	\end{Prop}
	
	\begin{proof}
		Let $x=(x',z)$ with $r=|x'|\geq4$ and write $y=(y',k)$. Since $|\omega|\leq C|\nabla u|$, the $D$-condition gives $\omega\in L^2(\R^3)$. By Lemma \ref{lem:Biot-Savart}\textup{(i)}, we have
		\begin{align*}
			|u(x)| &\leq C \int_{|y'|\leq1}\frac{|\omega(y',k)|}{|x'-y'|^2+(z-k)^2}\,dy'dk + C\int_{|y'|>1}\frac{|\omega(y',k)|}{|x'-y'|^2+(z-k)^2}\,dy'dk \\
			&=:C(I_0+I_1).
		\end{align*}
		
		For $I_0$, by
		\begin{equation*}
			\int_{\R}\frac{ds}{(a^2+s^2)^2}=\frac{\pi}{2a^3},
			\qquad a>0,
		\end{equation*}
		we have
		\begin{align}
			I_0
			\leq
			\norm{\omega}_{L^2(\R^3)}
			\left(
			\int_{|y'|\leq1}\int_{\R}
			\frac{dk\,dy'}{\left(|x'-y'|^2+(z-k)^2\right)^2}
			\right)^{1/2} \leq C r^{-3/2} \leq C r^{1-\beta}L(r)^{-\gamma},
			\label{eq:I0-bound}
		\end{align}
		where we used $|x'-y'|\geq r-1\geq r/2$ on $|y'|\leq1$. For $I_1$, by \eqref{eq:general-vorticity-decay}, Lemma \ref{lem:Riesz}, and
		\begin{equation*}
			\int_{\R}\frac{dk}{a^2+(z-k)^2}=\frac{\pi}{a},
			\qquad a>0,
		\end{equation*}
		we have
		\begin{align}
			I_1
			\leq
			C\int_{\R^2}\int_{\R}
			\frac{\Phi_{\beta,\gamma}(|y'|)}{|x'-y'|^2+(z-k)^2}\,dk\,dy'
			=
			C\int_{\R^2}\frac{\Phi_{\beta,\gamma}(|y'|)}{|x'-y'|}\,dy'
			\leq
			C r^{1-\beta}L(r)^{-\gamma}
			\label{eq:I1-bound}
		\end{align}
		Combining \eqref{eq:I0-bound} and \eqref{eq:I1-bound} proves \eqref{eq:general-velocity-decay}.
	\end{proof}

	\begin{proof}[Proof of Theorem \ref{thm:vorticity}]
		Apply Proposition \ref{prop:decay-transfer} with $\beta=\frac53$. The vorticity assumption \eqref{eq:vorticity-endpoint} gives
		\begin{equation*}
			\sup_{\substack{|x'|=r\\ z\in\R}}|u(x',z)|
			\leq
			\frac{C}{r^{2/3}[\log(\ee+r)]^\gamma},
			\qquad r\geq4.
		\end{equation*}
		Since $u$ is bounded, enlarging the constant extends the same estimate to $1\leq r<4$. Theorem \ref{thm:velocity} therefore applies, and $\gamma>\frac13$ yields $u\equiv0$.
	\end{proof}

	\appendix
	
	\section{Auxiliary estimates}
	\label{sec:appendix}

	The absolute-convergence argument in Lemma \ref{lem:Biot-Savart}\textup{(ii)} uses the next elementary estimate.

	\begin{Lem}
		\label{lem:potential-estimate}
		Let $1<a<2$ and $\beta\geq0$. Then for every $x'\in\R^2$ with $r=|x'|$,
		\begin{equation}
			\int_{\R^2}\frac{(1+|y'|)^{-a}[\log(\ee+|y'|)]^\beta}{|x'-y'|}\,dy'\leq C_{a,\beta}(1+r)^{1-a}[\log(\ee+r)]^\beta.
			\label{eq:two-dimensional-potential}
		\end{equation}
	\end{Lem}
	
	\begin{proof}
		For $0\leq r\leq2$, the integral is uniformly bounded. Indeed, the singularity $|x'-y'|^{-1}$ is locally integrable in two dimensions, while for large $|y'|$ one has $|x'-y'|\gtrsim |y'|$, so the tail is controlled by
		\[
		\int_1^\infty \rho^{-a}[\log(\ee+\rho)]^\beta\,d\rho<\infty,\quad a>1.
		\]
		Since
		\[
		(1+r)^{1-a}[\log(\ee+r)]^\beta\geq c_{a,\beta}>0,\qquad 0\leq r\leq2,
		\]
		the desired estimate follows in this range.
		
		Assume now that $r>2$ and split
		\[
		D_1=\{|y'|\leq r/2\},\qquad D_2=\{r/2<|y'|<2r\},\qquad D_3=\{|y'|\geq2r\}.
		\]
		On $D_1$, $|x'-y'|\geq r/2$, and therefore
		\[
		\int_{D_1}\frac{(1+|y'|)^{-a}[\log(\ee+|y'|)]^\beta}{|x'-y'|}\,dy'\leq Cr^{-1}\int_0^{r/2}(1+\rho)^{-a}[\log(\ee+\rho)]^\beta\rho\,d\rho\leq Cr^{1-a}[\log(\ee+r)]^\beta,
		\]
		where $a<2$ is used.
		
		On $D_2$, one has $(1+|y'|)^{-a}[\log(\ee+|y'|)]^\beta\leq Cr^{-a}[\log(\ee+r)]^\beta$, and $D_2-x'$ is contained in the disk $B_{3r}(0)$. Hence
		\[
		\int_{D_2}\frac{(1+|y'|)^{-a}[\log(\ee+|y'|)]^\beta}{|x'-y'|}\,dy'\leq Cr^{-a}[\log(\ee+r)]^\beta\int_{|h|<3r}\frac{dh}{|h|}\leq Cr^{1-a}[\log(\ee+r)]^\beta.
		\]
		
		Finally, on $D_3$ one has $|x'-y'|\geq |y'|/2$, and therefore
		\[
		\int_{D_3}\frac{(1+|y'|)^{-a}[\log(\ee+|y'|)]^\beta}{|x'-y'|}\,dy'\leq C\int_{2r}^\infty \rho^{-a}[\log(\ee+\rho)]^\beta\,d\rho\leq Cr^{1-a}[\log(\ee+r)]^\beta.
		\]
		Combining the three regions, we get \eqref{eq:two-dimensional-potential}.
	\end{proof}
	
	The following Riesz-potential bound can be inserted into the three-dimensional Biot--Savart representation.
	
	\begin{Lem}
		\label{lem:Riesz}
		Let \(1<\beta<2\), \(\gamma\geq0\), and $\Phi_{\beta,\gamma}(\rho)
		:=
		\mathbf{1}_{\{\rho\geq1\}}
		\rho^{-\beta}L(\rho)^{-\gamma}$. Then there exists a constant \(C=C(\beta,\gamma)>0\) such that, for every \(x'\in\R^2\) with \(r=|x'|\geq4\),
		\begin{equation*}
			\int_{\R^2}
			\frac{\Phi_{\beta,\gamma}(|y'|)}
			{|x'-y'|}
			\,dy'
			\leq
			Cr^{1-\beta}L(r)^{-\gamma}.
		\end{equation*}
	\end{Lem}
	
	\begin{proof}
		Let $F(s):=s^{2-\beta}L(s)^{-\gamma}$. A direct computation gives
		\[
		F'(s)
		=
		s^{1-\beta}L(s)^{-\gamma}
		\left(
		2-\beta
		-
		\frac{\gamma s}{(\ee+s)L(s)}
		\right).
		\]
		Then, there exists \(s_0=s_0(\beta,\gamma)\geq2\) such that
		\[
		F'(s)
		\geq
		\frac{2-\beta}{2}
		s^{1-\beta}L(s)^{-\gamma},
		\qquad s\geq s_0.
		\]
		Integrating this inequality and enlarging the constant to absorb the fixed interval \([1,s_0]\), we obtain
		\begin{equation}
			\int_1^R
			s^{1-\beta}L(s)^{-\gamma}\,ds
			\leq
			CR^{2-\beta}L(R)^{-\gamma},
			\qquad R\geq2.
			\label{eq:weighted-inner-inline}
		\end{equation}
		Similarly, setting $G(s):=s^{1-\beta}L(s)^{-\gamma}$, we have
		\[
		-G'(s)
		=
		s^{-\beta}L(s)^{-\gamma}
		\left(
		\beta-1
		+
		\frac{\gamma s}{(\ee+s)L(s)}
		\right)
		\geq
		(\beta-1)s^{-\beta}L(s)^{-\gamma}.
		\]
		Hence, integrating over \([R,\infty)\) yields
		\begin{equation}
			\int_R^\infty
			s^{-\beta}L(s)^{-\gamma}\,ds
			\leq
			CR^{1-\beta}L(R)^{-\gamma},
			\qquad R\geq2.
			\label{eq:weighted-outer-inline}
		\end{equation}
		Note that
		\begin{equation}
			L(\lambda r)\sim L(r),
			\quad
			\text{uniformly for }
			\lambda\in[1/2,2]
			\text{ and }r\geq4.
			\label{eq:log-comparability}
		\end{equation}

		Decompose \(\R^2\) into the four regions
		\begin{align*}
			E_1&:=\{|y'|\leq r/2\},\quad
			E_2:=\{|x'-y'|\leq r/2\},\\
			E_3:=\{r/2&<|y'|<2r, |x'-y'|>r/2\},\quad
			E_4:=\{|y'|\geq2r\}.
		\end{align*}
		These regions cover \(\R^2\), and their overlaps are harmless because the integrand is nonnegative. For \(E_1\), one has $|x'-y'|
		\geq
		r-|y'|
		\geq
		r/2$. Using polar coordinates, \eqref{eq:weighted-inner-inline}, and \eqref{eq:log-comparability}, we obtain
		\begin{align}
			\int_{E_1}
			\frac{\Phi_{\beta,\gamma}(|y'|)}
			{|x'-y'|}
			\,dy'
			\leq
			\frac{C}{r}
			\int_1^{r/2}
			\rho^{1-\beta}L(\rho)^{-\gamma}\,d\rho\leq
			Cr^{-1}
			\left(\frac r2\right)^{2-\beta}
			L\left(\frac r2\right)^{-\gamma} \leq
			Cr^{1-\beta}L(r)^{-\gamma}.
			\label{eq:E1}
		\end{align}
		On \(E_2\), $r\sim|y'|$. Therefore, changing variables \(h=x'-y'\), we obtain
		\begin{align}
			\int_{E_2}
			\frac{\Phi_{\beta,\gamma}(|y'|)}
			{|x'-y'|}
			\,dy'
			\leq
			Cr^{-\beta}L(r)^{-\gamma}
			\int_{|h|\leq r/2}
			\frac{dh}{|h|}\leq
			Cr^{1-\beta}L(r)^{-\gamma}.
			\label{eq:E2}
		\end{align}
		On \(E_3\), $|x'-y'|^{-1}\leq 2r^{-1}$ and $r\sim|y'|$. It follows that
		\begin{align}
			\int_{E_3}
			\frac{\Phi_{\beta,\gamma}(|y'|)}
			{|x'-y'|}
			\,dy'
			\leq
			\frac{C}{r}
			r^{-\beta}L(r)^{-\gamma}|E_3|\leq
			Cr^{1-\beta}L(r)^{-\gamma}.
			\label{eq:E3}
		\end{align}
		Finally, on \(E_4\), $|x'-y'|\geq |y'|/2$. Using polar coordinates, \eqref{eq:weighted-outer-inline}, and \eqref{eq:log-comparability}, we obtain
		\begin{align}
			\int_{E_4}
			\frac{\Phi_{\beta,\gamma}(|y'|)}
			{|x'-y'|}
			\,dy'
			\leq
			C\int_{2r}^\infty
			\rho^{-\beta}L(\rho)^{-\gamma}\,d\rho\leq
			C(2r)^{1-\beta}L(2r)^{-\gamma}\leq
			Cr^{1-\beta}L(r)^{-\gamma}.
			\label{eq:E4}
		\end{align}
		
		Combining \eqref{eq:E1}--\eqref{eq:E4} completes the proof.
	\end{proof}

	\section*{Declarations}
	\begin{itemize}
		\item \textbf{Acknowledgments} 
		W. Wang was supported by National Key R\&D Program of China (No.2023YFA1009200) and NSFC under grant 12471219.
		\item \textbf{Conflict of interest} The authors declare that they have no conflict of interest.
		\item \textbf{Data Availability} Data sharing is not applicable to this article as no datasets were generated or analyzed during the current study.
	\end{itemize}
	
	\bibliographystyle{myamsalpha}
	\bibliography{SNS-Liouville-WY2026-refs}

\end{document}